\documentclass[12pt]{amsart}

\usepackage{mathtools}
\mathtoolsset{showonlyrefs,showmanualtags} 

\usepackage{amssymb, amsmath, amsthm, esint}
\usepackage{mathrsfs}
\usepackage{bbm,dsfont}
\usepackage{hyperref}
\usepackage{mathtools}
\usepackage{fullpage}
\usepackage{color}
\usepackage{tikz}

\usepackage[backrefs]{amsrefs}

 \usepackage[top=1.5in, bottom=1.5in, left=1.25in, right=1.25in]	{geometry}
\usepackage[all]{xy}
\usepackage{amsmath}
\usepackage{amssymb}
\usepackage{amsthm}
\usepackage{amscd}

\newtheorem{theorem}[equation]{Theorem}
\newtheorem{definition}[equation]{Definition}
\newtheorem{proposition}[equation]{Proposition}
\newtheorem{cor}[equation]{Corollary}
\newtheorem{lemma}[equation]{Lemma}

\newtheorem{remark}[equation]{Remark}

\usepackage{hyperref} 
\hypersetup{
    colorlinks=true,       
    linkcolor=blue,          
    citecolor=magenta,        
    filecolor=magenta,      
    urlcolor=cyan           
}

\begin{document}
\title{A Hard-Analytic Proof of ``Most" Polynomial Wiener-Wintner Theorems for Infinite Measure Spaces}

\author{Ben Krause}
\address{
Department of Mathematics,
University of Bristol \\
Beacon House, Queens Rd, Bristol BS8 1QU}
\email{ben.krause@bristol.ac.uk}

\date{\today}

\maketitle

\begin{abstract}
    We provide a new proof of ``most" cases of the polynomial Wiener-Wintner theorem for $\sigma$-finite spaces, using hard-analytic methods. Specifically, we prove that whenever $(X,\mu,T)$ is a $\sigma$-finite measure-preserving system, and $f \in L^p(X), \ 1 \leq p < \infty$, there exists a co-null set $X_f \subset X$ so that for all $\omega \in X_f$
\[ \frac{1}{N} \sum_{n \leq N} e^{2 \pi i P(n)} f(T^n \omega) \]
converges for all polynomials $P$ which are either linear, or vanish to degree $2$ at the origin.
\end{abstract}

\section{Introduction}

The Wiener-Wintner ergodic theorem \cite{WW} is a classical generalization of Birkhoff's Theorem \cite{BI}; here and throughout, by a \emph{measure-preserving system}, $(X,\mu,T)$, we mean a probability space, $(X,\mu)$, equipped with a measure-preserving transformation
$T : X \to X$, so that
\[ \mu(T^{-1} E) = \mu(E) \text{ for all } E \subset X \text{ measurable};\]
a $\sigma$-finite measure-preserving system is as above, but the underlying space $(X,\mu)$ is $\sigma$-finite.

\begin{theorem}[Wiener-Wintner Ergodic Theorem]\label{t:WW}
Let $(X,\mu,T)$ be a measure-preserving system, and let $f \in L^1(X)$ be arbitrary. Then there exists a subset $X_f \subset X$ with $\mu(X_f) = 1$ so that for all $\omega \in X_f$
\[ \lim_N \frac{1}{N} \sum_{n \leq N} e^{2\pi i n \theta} f(T^n \omega)  \]
exists for {all} $\theta \in [0,1]$.
\end{theorem}

The significance of this theorem is that the limit holds for \emph{all} $\theta$, rather than just for almost every $\theta$; this weaker statement follows from applying Birkhoff's Ergodic Theorem in the product setting. And, whenever $(X,\mu)$ is countably generated, one can choose a universal full-measure subset, $X_0 \subset X$ so that the conclusion of Theorem \ref{t:WW} holds for all $f \in L^1(X)$ whenever $\omega \in X_0$; this follows from the maximal ergodic theorem and a standard density argument.

At this point the Wiener-Wintner Theorem is well understood, and is known to accommodate more general weighted averages including polynomials \cite{LES}, nilsequences \cite{HK}, Hardy field functions \cite{EK}, etc.; see \cite{ASSBook,AssSum} for a fuller discussion. Moreover, this line of inquiry ultimately admits the profound extension to the \emph{return times} setting \cite{B2}.

In this note we provide a hard-analytic proof of ``most" cases of the polynomial Wiener-Wintner Theorem which holds in the $\sigma$-finite setting; see \cite{ASSBook} for an alternative proof of the case with linear modulations.

\begin{theorem}\label{t:main}
Let $(X,\mu,T)$ be a $\sigma$-finite measure-preserving system, and let $f \in L^p(X), \ 1 \leq p < \infty$ be arbitrary. Then there exists a co-null subset $X_f \subset X$ so that for all $\omega \in X_f$
\[ \lim_N \frac{1}{N} \sum_{n \leq N} e^{2\pi i P(n)} f(T^n \omega)  \]
exists for {all} polynomials $P \in \mathbb{R}[\cdot]$ that are either linear, or vanish to degree two at the origin.
\end{theorem}

Standard arguments allow one to replace the character $t \mapsto e^{2\pi i t}$ with any Riemann integrable function $\phi : \mathbb{T} \to \mathbb{C}$ without adjusting $X_f$. And, we may remove the assumption that $X$ is $\sigma$-finite by restricting to the ($\sigma$-finite) $\sigma$-algebra generated by $\{ f(T^n \cdot) : n \in \mathbb{Z} \}$ for each individual $f \in L^p(X), \ 1 \leq p < \infty$.

\subsection{Proof Overview}
In proving Theorem \ref{t:main}, by the Maximal Ergodic Theorem and a standard density argument, we may assume restrict to simple functions, and then by linearity to integrable indicators; with this in mind, we may restrict as well to lacunary times $N \in \{ \lfloor \lambda^k \rfloor : k \in \mathbb{N} \}$ for $1 < \lambda \leq 2$. And, there is no harm in replacing the rough cut-off $\frac{1}{N} \mathbf{1}_{[1,N]}$ with a smooth bump function
\begin{align}
    \varphi_N(n) := \frac{1}{N} \varphi(\frac{n}{N}) 
\end{align}
where 
\[ \| \varphi - \mathbf{1}_{[0,1]} \|_{L^1(\mathbb{R})} \leq \epsilon_0 \]
and
\begin{align}\label{e:eps} |\partial^\alpha \varphi| \lesssim_\alpha \epsilon_0^{-\alpha} \cdot \mathbf{1}_{[0,1]} \end{align}
for sufficiently many $\alpha$; see Subsection \S \ref{sss:O} below to recall the $\lesssim_\alpha$ notation. We will regard $\epsilon_0$ as fixed throughout the remainder of the paper. Finally, we will restrict as we may to polynomials whose degree is bounded above by $d$.

Thus, if we define
\begin{align}
\Phi_M^P f(\omega) := \Phi_M^{P;T} f(\omega) := \sum_m \varphi_M(m) e^{2 \pi i P(m)}  f(T^m \omega)
\end{align}
and for each $\omega \in X$, define the (truncated) \emph{jump-counting} function
\begin{align}\label{e:jump0}
    &N_{\tau,H}^P f(\omega) := N_{\tau,H}^{P;T} f(\omega) := \\
    &\qquad := \sup\Big\{ K : \text{ there exists } M_0(\omega) < M_1(\omega)< \dots < M_K(\omega) \leq H :\\ & \qquad \qquad\qquad \qquad  |\Phi_{M_{k-1}(\omega)}^P f(\omega) - \Phi_{M_k(\omega)}^P f(\omega)| \gg \tau \Big\} 
\end{align}
where all times are of the form $\lfloor (1+\tau)^{\mathbb{N}} \rfloor$, the following elementary lemma will guide our analysis; see Subsection \S \ref{sss:O} below for a review of little-Oh notation.

\begin{lemma}\label{l:trans1}
    Suppose that, for each integrable indicator $f$,
    \begin{align}
        \mu(\{X : \sup_{P \in \mathcal{P}} N_{\tau,H}^P f(\omega) \geq L \}) = o_{L \to \infty;\tau,\| f \|_{L^2(X)}}(1),
    \end{align}
    independent of $H$; then for any $f \in L^p(X), \ 1 \leq p < \infty$, there exists a co-null set $X_f \subset X$ so that for all $\omega \in X_f$
    \begin{align}
        \lim_N \sum_{n} \varphi_N(n) e^{2 \pi i P(n)} f(T^n \omega)
    \end{align}
exists for all $P \in \mathcal{P}$.
\end{lemma}
\begin{proof}
As above, it suffices to prove convergence only for integrable indicators; let $f$ be an arbitrary such indicator.

Since 
\[ P \mapsto N_{\tau,H}^P f(\omega) \]
is continuous for each $\omega, H,\tau$ and bounded $f$, it suffices to restrict to polynomials with rational coefficients, so 
\[ \{X : \sup_{P \in \mathcal{P}} N_{\tau,H}^P f(\omega) \geq L \} \]
is measurable for each $\tau, H$, and indicator function $f$. We proceed by contradiction. So, suppose that there exists a nontrivial set $E \subset X$ with $\mu(E) \gg \tau$ so that for all $\omega \in E$ there exists $P_\omega \in \mathcal{P}$ so that
\[ \limsup_{N,M \to \infty} |\Phi_M^{P_\omega} f(\omega) - \Phi_N^{P_\omega} f(\omega)| \gg \tau; \]
there is no loss of generality in assuming that $M,N$ lie in a $(1+\tau)$-lacunary sequence.

So, for all $\omega \in E$, there exists an infinite sequence $M_0(\omega) < M_1(\omega) < \dots $ so that
\begin{align}
     |\Phi_{M_{k-1}(\omega)}^{P_\omega} f(\omega) - \Phi_{M_k(\omega)}^{P_\omega} f(\omega)| \gg \tau, \; \; \; k \geq 1;
\end{align}
possibly after replacing $E$ with a subset, $E' \subset E$ of measure $\gg \tau$, we may assume that for all $\omega \in E'$, there exists a subsequence
\[ M_0(\omega) < M_1(\omega) < \dots < M_K(\omega) \leq H \]
with $K \geq L \gg_\tau 1$ so that the above oscillatory estimate holds; while $H$ may depend on $\tau, L$, our estimates will be independent of this precise choice of $H$. In particular, for each such $\omega$,
\[ \sup_{P \in \mathcal{P}} N_{\tau,H}^P f(\omega) \geq L,\]
so that
\[ \tau \ll \mu(E') \leq \mu(\{X : \sup_{P \in \mathcal{P}} N_{\tau,H}^P f(\omega) \geq L \}) = o_{L \to \infty;\tau,\| f \|_{L^2(X)}}(1),\]
for the desired contradiction.
\end{proof}

We will apply Lemma \ref{l:trans1} by dominating the pertaining {jump-counting} function \eqref{e:jump0} by its \emph{variational} counterpart,
\begin{align}\label{e:toyVr}
    V^r_{\mathcal{P}}f := \sup_{P \in \mathcal{P}} \sup \big( \sum_i |\Phi_{M_k}^P f - \Phi_{M_{k-1}}^P f|^r \big)^{1/r}, \; \; \; 2< r < \infty
\end{align}
where the inner supremum is over finite increasing subsequences, as we may bound
\begin{align}
    \sup_{P \in \mathcal{P}} \tau N_{\tau,H}^P f(\omega)^{1/r} \leq V^r_{\mathcal{P}}f(\omega)
\end{align}
for each $\tau,H$. In particular, we will be interested in proving the norm estimates
\begin{align}
    \| V^r_{\mathcal{P}} f \|_{L^2(X)} \lesssim_{r,\epsilon_0,\lambda,d} \|f\|_{L^2(X)} 
\end{align}
whenever the sequence of times in the definition of $V^r_{\mathcal{P}}$ live inside a lacunary sequency, $\{ M_k \} \subset \{ \lfloor \lambda^n \rfloor \}, \ 1 < \lambda \leq 2$;
by Calder\'{o}n's transference principle, \cite{C1}, we may conduct our analysis on the integer lattice, and focus on estimating
\begin{align}\label{e:toyVr}
    \mathcal{V}^r_{\mathcal{P}}f := \sup_{P \in \mathcal{P}} \sup \big( \sum_i |A_{M_k}^P f - A_{M_{k-1}}^P f|^r \big)^{1/r}, \; \; \; 2< r < \infty
\end{align}
on $\ell^2(\mathbb{Z})$, where
\begin{align}
    A^P_M f(x) := \sum_n \varphi_M(n) e^{2 \pi i P(n)} f(x-n) : \mathbb{Z} \to \mathbb{C}
\end{align}
are polynomially modulated discrete convolution operators.

At this point, the analysis splits into two cases: in the first case, when $\mathcal{P}$ consists of linear polynomials, \eqref{e:toyVr} encodes a modulation invariance, namely
\begin{align}
    \mathcal{V}^r_{\mathcal{P}} f \equiv \mathcal{V}^r_{\mathcal{P}} f_\theta, \; \; \; f_\theta(n) := e^{2 \pi i n \theta} f(n)
\end{align}
which necessitates an approach deriving from time frequency analysis. Fortunately, the relevant analysis was developed in deep work of Oberlin-Seeger-Tao-Thiele-Wright \cite{O+}, in the context of quantitative convergence of Fourier series. We make no comments on their work, which goes by way of wave packet analysis, but simply import and appropriately transfer the estimate
\begin{align}\label{e:O+Vr}
   &\| \sup_\theta \sup \big( \sum_k |\int_{-r_k}^{r_k} f(x-t) e^{2 \pi i \theta t} \frac{dt}{2r_k} - \int_{-r_{k-1}}^{r_{k-1}} f(x-t) e^{2 \pi i \theta t} \frac{dt}{2r_{k-1}}|^r \big)^{1/r} \|_{L^2(\mathbb{R})} \\
   &\lesssim_r  \|f \|_{L^2(\mathbb{R})}, \; \; \; r > 2,
\end{align}
see \cite[Remark D.4]{O+}; the inner supremum is over finite increasing subsequences.

The second case we address is when $\mathcal{P}$ is given by
\begin{align}\label{e:Pd}  \mathcal{P}_d := \{ P \in \mathbb{R}[\cdot] : P(0) = P'(0) =0, \ \text{deg}(P) \leq d \}, \end{align}
polynomials of bounded degree which vanish to degree two at the origin. In this context, no modulation invariance presents, and orthogonality methods become more relevant; on the other hand, the lack of linearity in the class of $\mathcal{P}_d$ eliminates the connection to the real-variable problem. We build off prior work \cite{KSW} to prove
\begin{align}
    \| \mathcal{V}^r_{\mathcal{P}}f \|_{\ell^2(\mathbb{Z})} \lesssim_{\epsilon_0,\lambda,d}  (\frac{r}{r-2})^2 \| f\|_{\ell^2(\mathbb{Z})},
\end{align}
which in fact implies 
\begin{align}
    \| \mathcal{V}^r_{\mathcal{P}}f \|_{\ell^p(\mathbb{Z})} \lesssim_{\epsilon_0,\lambda,d} (\frac{r}{r-2})^2 \| f\|_{\ell^p(\mathbb{Z})}, \; \; \; r > \max\{p,p'\},
\end{align}
by interpolating against a special case of \cite{KSW}. While the arguments of \cite{KSW} lived at the interface of analytic number theory and harmonic analysis, deriving from Bourgain's celebrated work on the polynomial ergodic theorems \cite{B0,B1,B2}, our work here is purely analytic, as the number theoretic input of \cite{KSW} can be imported directly; our main ingredient is a ``variable coefficient" multi-frequency variation estimate, see \cite{GZK} for similar work.

Finally, we remark that addressing the case of general polynomial modulations using these methods is out of range of current techniques, as the polynomial modulation invariance whenever $\mathcal{P}$ consist of all real variable polynomials of degree $\leq d$,
\begin{align}
    \mathcal{V}^r_{\mathcal{P}} f \equiv \mathcal{V}^r_{\mathcal{P}} f_P, \; \; \; f_P(n) := e^{2 \pi i P(n)} f(n), \; \; \; \text{deg}(P) \leq d,
\end{align}
introduces a degree of complexity similar to that found in addressing the pointwise convergence of the multiple ergodic averages
\[ \frac{1}{N} \sum_{n \leq N} T^n f_1 \cdot \dots \cdot T^{n (d+1)} f_{d+1}, \; \; \; f_i \in L^\infty(X);\]
we hope to address this issue in future work. 
\medskip

The structure of the paper is as follows:

In \S \ref{s:lin}, we quickly address the case of linear polynomials;

In \S \ref{s:poly0}, we turn to the case of oscillatory polynomials, and complete the proof of Theorem \ref{t:main} by developing the necessary variational theory.

\subsection{Notation}\label{ss:not}
We use
\[ e(t) := e^{2 \pi i t} \]
throughout to denote the complex exponential, and let
\[ M_{\text{HL}} f(x) := \sup_{N \geq 0} \frac{1}{2N + 1} \sum_{|n| \leq N} |f(x-n)| \]
denote the discrete Hardy-Littelwood maximal function.

For a sequence of scalars, $\{ a_N \}$, we define the $r$-variation
\begin{align}
    \mathcal{V}^r(a_N:N) := \sup \big( \sum_k |a_{N_k} - a_{N_{k-1}}|^r\big)^{1/r},
\end{align}
where the supremum is over all finite increasing subsequences. We will only be interested below in the case where $2 < r < \infty$.

For functions $\{ f_N \} : X \to \mathbb{C}$, we define
\begin{align}
    \mathcal{V}^r(f_N:N)(x) := \mathcal{V}^r(f_N(x) :N)
\end{align}
pointwise.


\subsubsection{Asymptotic Notation}\label{sss:O}
We will make use of the modified Vinogradov notation. We use $X \lesssim Y$ or $Y \gtrsim X$ to denote
the estimate $X \leq CY$ for an absolute constant $C$ and $X, Y \geq 0.$  If we need $C$ to depend on a
parameter, we shall indicate this by subscripts, thus for instance $X \lesssim_p Y$ denotes the estimate $X \leq C_p Y$ for some $C_p$ depending on $p$. We use $X \approx Y$ as shorthand for $Y \lesssim X \lesssim Y$. We use the notation $X \ll Y$ or $Y \gg X$ to denote that the implicit constant in the $\lesssim$ notation is extremely large, and analogously $X \ll_p Y$ and $Y \gg_p X$.

We also make use of big-Oh and little-Oh 
notation: we let $O(Y)$  denote a quantity that is $\lesssim Y$ , and similarly
$O_p(Y )$ will denote a quantity that is $\lesssim_p Y$; we let $o_{t \to a}(Y)$
denote a quantity whose quotient with $Y$ tends to zero as $t \to a$ (possibly $\infty$), and
$o_{t \to a;p}(Y)$
denote a quantity whose quotient with $Y$ tends to zero as $t \to a$ at a rate depending on $p$.

\section{Linear Modulations}\label{s:lin}
We organize our analysis in this section by introducing the following definition:

\begin{definition}
    A function $\chi : \mathbb{R} \to \mathbb{C}$ is said to satisfy $r$-Variational Carleson with constant $C$ if the following inequality holds, independent over all choices of (finitely many) measurable functions $\{ R_i \}, \theta$:
    \begin{align}
&\| \Big( \sum_i |\int \big( \widehat{\chi}(R_i(x)(\beta - \theta(x))) - \widehat{\chi}(R_{i+1}(x)(\beta - \theta(x))) \big) \hat{f}(\beta) e(\beta x)|^r \Big)^{1/r} \|_{L^2(\mathbb{R})} \\
& \qquad \leq C \frac{r}{r-2} \|f\|_{L^2(\mathbb{R})}.
    \end{align}
\end{definition}

With this in mind, Remark D.4 of \cite{O+} can be stated as follows.
\begin{proposition}[Variational Carleson, Averaging Formulation]\label{p:VCar}
$\mathbf{1}_{[-1/2,1/2]}$ satisfies $r$-Variational Carleson with constant $O(1)$.
\end{proposition}

Our task, therefore, is to prove the following proposition.

\begin{proposition}\label{p:vCar0}
    Suppose that $\varphi$ is smooth, and satisfies
    \begin{align}
        \| \partial^\alpha \varphi \|_\infty \leq \epsilon^{-\alpha} \cdot \mathbf{1}_{[0,1]}
    \end{align}
    for sufficiently many $\alpha$, and that all times $M \in \lfloor \lambda^{\mathbb{N}} \rfloor \subset \mathbb{N}, \ \lambda > 1$. Then
    \begin{align}
      \|  \sup_\theta \mathcal{V}^r( \sum_m \varphi_M(m) g(a+m) e(m \theta) : M) \|_{\ell^2(\mathbb{Z})} \lesssim \epsilon^{-O(1)} \frac{r}{r-2} \frac{\lambda}{\lambda -1} \| f\|_{\ell^2(\mathbb{Z})}. 
    \end{align}
\end{proposition}
The proof of Proposition \ref{p:vCar0} will derive from Proposition \ref{p:VCar} by way of a transference argument. We begin with a convexity result.

\begin{lemma}
    Suppose that $\varphi \in \mathcal{C}^1(\mathbb{R}) \cap \mathcal{C}_0(\mathbb{R})$. Then $\varphi$ satisfies $r$-Variational Carleson with constant $O(C_\varphi)$, where
    \[ C_\varphi :=  \| x \varphi'(x) \|_{L^1(\mathbb{R})}.\]
\end{lemma}
\begin{proof}
By the fundamental theorem of calculus we may express
\begin{align}
    \phi( t) = -\int \frac{1}{s} \mathbf{1}_{[0,s]}( t) \ s \phi'(s) \ ds,
\end{align}
and thus whenever $\phi$ is even
\begin{align}
    \phi(t) = - \int \frac{1}{2s} \mathbf{1}_{[-s,s]}(t) \ s \phi'(s) \ ds,
\end{align}
so 
\begin{align}
&    \int \big( \widehat{\phi}(R_i(x)(\beta - \theta(x))) - \widehat{\phi}(R_{i+1}(x)(\beta - \theta(x))) \big) \hat{f}(\beta) e(\beta x) \ d\beta\\
& \qquad = - \int \Big( 
\int \big( \widehat{\chi}(sR_i(x)(\beta - \theta(x))) - \widehat{\chi}(sR_{i+1}(x)(\beta - \theta(x))) \big) \hat{f}(\beta) e(\beta x) \Big) \ s \phi'(s) \ ds,
\end{align}
where $\chi = \frac{1}{2} \mathbf{1}_{[-1,1]}$, from which the result follows.
\end{proof}

Since, for the purpose of convergence, we are only interested in lacunary sequences, i.e.\ sequences of the form
\[ \{ M_i \} \subset \lfloor \lambda^{\mathbb{N}} \rfloor, \;\; \; \lambda >1,\]
we introduce the following definition.
\begin{definition}
    A function $\chi : \mathbb{R} \to \mathbb{C}$ is said to satisfy lacunary $r$-Variational Carleson with constant $C$ if the following inequality holds, independent over all choices of (finitely many) measurable functions $\{ R_i \} \subset \lfloor \lambda^{\mathbb{N}} \rfloor, \theta$:
    \begin{align}
&\| \Big( \sum_i |\int \big( \widehat{\chi}(R_i(x)(\beta - \theta(x))) - \widehat{\chi}(R_{i+1}(x)(\beta - \theta(x))) \big) \hat{f}(\beta) e(\beta x)|^r \Big)^{1/r} \|_{L^2(\mathbb{R})}\\
&\qquad \leq C \frac{\lambda}{\lambda - 1} \frac{r}{r-2} \|f\|_{L^2(\mathbb{R})}.
    \end{align}
\end{definition}

We next introduce a quantity that will allow us to pass from one function which satisfies lacunary $r$-Variational Carleson to another:

Given two real-variable functions $\varphi,\chi$, define
\begin{align}
    A[\varphi,\chi]&:= \sup_{\{R\}, \ 0 < s \leq S < \infty} |\sum_{R} \int_{s \leq |t| \leq S}  \varphi(t/R) - \chi(t/R) \ \frac{dt}{R}|  \\
    & \qquad + \sup_{\{R\}, \ t \neq 0} \big( |t|/R \cdot \sum_R |\varphi(t/R) - \chi(t/R)| \big) + \sup_{R, \ t \neq 0} \big( {|t|^2}/{R^2} \cdot \sum_R |\varphi'(t/R) - \chi'(t/R)| \big)\end{align}
where the supremum is over every $2$-lacunary sequence of times, $\{R \} \subset \mathbb{N}$.

\begin{lemma}
Suppose that $\chi$ satisfies lacunary $r$-Variational Carleson constant $C$. Then $\varphi$ satisfies lacunary $r$-Variational Carleson with constant $C + O(A[\varphi,\chi])$.
\end{lemma}
\begin{proof}
Set $\psi := \varphi - \chi$; it suffices to bound
\begin{align}
    \| \sup_\theta \big( \sum_R |\int \psi(R(\beta - \theta)) \hat{f}(\beta) e(\beta x) \ d\beta|^2 \big)^{1/2} \|_{L^2(\mathbb{R})} \lesssim A[\phi,\chi] \| f \|_{L^2(\mathbb{R})}
\end{align}
for any $2$-lacunary sequence $\{ R_i\}$. The result follows from randomization and Carleson's Theorem for Calder\'{o}n-Zygmund Kernels, a result implicitly contained in e.g.\ \cite{Fef}, see also \cite{B+} or \cite{Lie} for the explicit result along with a significant strengthening.
\end{proof}

We now begin transfering to the discrete context.
\begin{lemma}
    Suppose that $\varphi$ satisfies lacunary $r$-variational Carleson with constant $C$, and has Fourier support inside $\{|\xi| \leq \epsilon^{-1} \}$, and that all times $M \in \lfloor \lambda^{\mathbb{N}} \rfloor  \subset \mathbb{N}$. Then
    \begin{align}
      \|  \sup_\theta \mathcal{V}^r( \sum_m \varphi_M(m) g(a+m) e(m \theta) : M) \|_{\ell^2(\mathbb{Z})} \lesssim C \epsilon^{-1} \frac{r}{r-2}  \frac{\lambda}{\lambda -1} \| g\|_{\ell^2(\mathbb{Z})}. 
    \end{align}
\end{lemma}
\begin{proof}
By convexity and lacunarity we can assume that $M \geq 2^{\epsilon^{-1}}$.
If we bound
\begin{align}
    \mathcal{V}^r( \sum_m \varphi_M(m) g(a+m) e(m \theta) : M) \leq \sum_{0 \leq l \leq 4} \sup_{l/5 \leq \theta < (l+1)/5} \mathcal{V}^r (\sum_m \varphi_M(m) g(a+m) e(m \theta)  : M )
\end{align}
then it suffices to each summand in $l$ individually; by monotone convergence, it suffices to estimate the $\ell^2$ norms of
\[ \sup_{\theta \in [l/5,(l+1)/5) \cap \frac{1}{T} \cdot \mathbb{Z}} \mathcal{V}^r( \sum_m \varphi_M(m) g(a+m) e(m \theta) : M \leq T), \]
independent of $T$; the result then follows from Magyar-Stein-Wainger transference \cite[Proposition 2.1]{MSW} and the hypothesis.
\end{proof}

The proof of Proposition \ref{p:vCar0} now follows from a standard dyadic decomposition.
\begin{proof}[Proof of Proposition \ref{p:vCar0}]
By the Riemann-Lebesgue lemma, $\varphi$ is smooth and satisfies
\begin{align}
    |\widehat{\varphi}(\xi)| \lesssim_A (1 + \epsilon |\xi|)^{-A}
\end{align}
for sufficiently large $A$; set 
\[ \varphi^{\text{E}}(t) := \frac{\varphi(t) + \varphi(-t)}{2},\] 
which satisfies the same Fourier decay, and decompose
\begin{align}
\varphi := \varphi_0 + \sum_{j \geq 1} \varphi_j, \; \; \; \varphi^{\text{E}} = \varphi^{\text{E}}_0 + \sum_{j \geq 1} \varphi^{\text{E}}_j
\end{align}
where
\[ \widehat{\varphi_j}(\xi) := \widehat{\varphi}(\xi) \cdot \psi(\xi \epsilon^2/2^j), \]
and $\varphi_0$ is defined by subtraction,
and similarly for $\varphi^{\text{E}}_j$, where 
\begin{align}
    \sum_{k} \psi(\xi/2^k) = \mathbf{1}_{\xi \neq 0}
\end{align}
and $\psi$ is a smooth approximation to $\mathbf{1}_{|\xi| \approx 1}$.

Note that $\varphi^{\text{E}}_0$ satisfies $r$-Variational Carleson with constant $O(\epsilon^{-O(1)})$, and similarly
\[ A[\varphi_0,\varphi^{\text{E}}_0] \lesssim \epsilon^{-O(1)},\] so $\varphi_0$ satisfies the lacunary $r$-Variational Carleson with constant $O(\epsilon^{-O(1)})$. And, for each $j \geq 1$, $\widehat{\varphi_j}, \widehat{\varphi^{\text{E}}_j}$ are supported in $\{ |\xi| \lesssim \epsilon^{-2} 2^j \}$, and that
\[ C_{\varphi^{\text{E}}_j}, \ A[\varphi_j,\varphi^{\text{E}}_j] \lesssim \epsilon^{10} 2^{-10 j} \]
(say), so the result follows from the triangle inequality.
\end{proof}

This concludes our work on linear modulations; for the remainder of the paper, all polynomials will be assumed to vanish to degree two at the origin, see \eqref{e:Pd}. 

\section{Polynomial Wiener Wintner}\label{s:poly0}

It will be convenient to adopt a singular integral perspective. Thus, fix a constant of lacunarity $1 < \lambda \leq 2$, and set
\[ \psi(t) := \varphi(t) - \lambda^{-1} \varphi(t/\lambda), \; \; \; \psi_k(t) := \lambda^{-k} \psi(\lambda^{-1} t),\]
so that $\{ \psi_k \}$ are mean-zero, supported in $\{ \lambda^k/10 \leq |x| \leq 10 \lambda^k \}$ and satisfy
\begin{align}
    \frac{\|{\psi_k}\|_{L^\infty(\mathbb{R})}}{\lambda^k} + \frac{\|\psi_k'\|_{L^\infty(\mathbb{R})}}{\lambda^{2k}} \leq C < \infty
\end{align}
uniformly in $k$; consolidate
\begin{align}
    \Psi_k(t) := \sum_{1 \leq j \leq k} \psi_j(t), \; \; \; 
    \Psi_k^s(t) := \sum_{2^{s/A_0} \leq j \leq k} \psi_j(t).
\end{align}

By telescoping appropriately and applying Calder\'{o}n's transference principle as above, our task is to prove the following proposition.

\begin{proposition}\label{p:var00}
The following estimate holds:
\begin{align}
    \| \mathcal{V}^r_d f \|_{\ell^2(\mathbb{Z})} \lesssim_{\epsilon_0,\lambda,d} (\frac{r}{r-2})^2 \|f \|_{\ell^2(\mathbb{Z})},
\end{align}
where
\begin{align}\label{e:vrd}
\mathcal{V}^r_d f(x) := \sup_{P \in \mathcal{P}_d} \sup \Big( \sum_{i} 
    |\sum_{n} \big( \Psi_{k_i}(n) - \Psi_{k_{i-1}}(n) \big)  e(P(n)) f(x-n)|^r \Big)^{1/r}
\end{align}
where the inner supremum is over all finite increasing subsequence $\{ k_i \}$.
\end{proposition}
\begin{remark}
    By interpolating this proposition with the $r = \infty$ version addressed in \cite{KSW}, the following norm estimates present:
    \begin{align}
        \| \mathcal{V}^r_d f \|_{\ell^p(\mathbb{Z})} \lesssim_{\epsilon_0,\lambda,d,p} (\frac{r}{r-2})^2 \|  f\|_{\ell^p(\mathbb{Z})}, \; \; \; r > \max\{p,p'\}.
    \end{align}
\end{remark}

\subsection{Preliminaries}
Let $A_0 \in \mathbb{N}$ be large but fixed, and for $P \in \mathcal{P}_d$, abbreviate
\begin{align}
    P_{\vec{\lambda}}(t) := \sum_{j=2}^d \lambda_j t^j.
\end{align}
For $P \in \mathbb{R}[\cdot]$, we use
\[ N_{2^k}(P)\]
to denote the \emph{coefficient norm} of $P$ at scale $2^k$, see \cite[Definition 3.1]{KSW}. We will use $\vec{\lambda}$ to denote elements of $\mathbb{T}^{d-1}$, $\vec{\lambda} = (\lambda_2,\dots,\lambda_d)$, and similarly use the notation
\[ \vec{A} := (A_2,\dots,A_d,Q) \]
where $\vec{A} \in \mathbb{Z}^{d-1}$ so that in particular $(\vec{A},Q) = 1$ means that $(A_2,\dots,A_d,Q) = 1$.

We apply the decomposition of \cite[Proposition 4.2]{KSW}; specifically, we have the following lemma.

\begin{lemma}\label{l:gettingstarted}
The following pointwise bound holds:
\begin{align}
\mathcal{V}^r_d f \leq \sum_{s \geq 1} \mathcal{A}_s f + \mathcal{E} f + O(M_{\text{HL}}f) + \mathcal{H}^r f,
\end{align}
where $\mathcal{E}$ is exactly as in \cite[Proposition 4.2]{KSW}, with bounded $\ell^2(\mathbb{Z})$ operator norm, $\mathcal{H}^rf$ is the $r$-variation of a truncated singular integral (whose $\ell^2$-operator norm is $O(\frac{r}{r-2})$), and
\begin{align}
    \mathcal{A}_s f(x) := \mathcal{V}^r \big( \sum_{k=2^{s/A_0}}^{k_0} \sum_m \psi_k(m) e(P_{\vec{\lambda}(x)}(m))
    f(x-m) \cdot \mathbf{1}_{n : N_{2^k}(P_{\vec{\lambda}}(n)) = 2^s}(x) : k_0 \geq 2^{s/A_0} \big).
\end{align}
\end{lemma}

In particular, to prove Proposition \ref{p:var00} it suffices to prove that 
\begin{align}\label{e:As}
\| \mathcal{A}_s f \|_{\ell^2(\mathbb{Z})} \lesssim_{\epsilon_0,\lambda,d} (\frac{r}{r-2})^2 2^{-cs} \|f \|_{\ell^2(\mathbb{Z})} \text{ for some } c > 0.
\end{align}

Next, by \cite[Proposition 5.6]{KSW}, we may replace $\mathcal{A}_sf$ with
\begin{align}
    \mathcal{V}^r_{s,d} f  := \sup_{\vec{\lambda}} \mathcal{V}^r( (L^s_{J,\vec{\lambda}})^{\vee}*f : J )
\end{align}
where
\[ \{ L_{J,\vec{\lambda}}^s \}\]
are Fourier multipliers, defined below:
\begin{align}
    L_{J,\vec{\lambda}}^s(\beta) = \sum_{(\vec{A},Q) = 1: 2^{s-1} \leq Q < 2^s} \sum_{B \leq Q} S(\vec{A}/Q,B/Q) &\Phi_{J,\vec{\lambda} - \vec{A}/Q}^*(\beta - B/Q) \chi_s(\beta - B/Q) \\
    & \qquad \times \mathbf{1}_{\| \lambda_j - A_j/Q \|_{\mathbb{T}} \leq 2^{-10s-10}, \ 2 \leq j \leq d}
\end{align}
where
\begin{align}
S(\vec{A}/Q,B/Q) := \frac{1}{Q} \sum_{r \leq Q} e(-\frac{A_2 r^2 + \dots + A_d r^d + r B}{Q})    
\end{align}
are complete Weyl sums; 
\begin{align}
        \Phi_{J,\vec{\lambda}}^{\vee}(x) = \Psi^s_J(x) e(-P_{\vec{\lambda}}(x)) \cdot \mathbf{1}_{\| \lambda_k \|_{\mathbb{T}} \leq J^{A_0} 2^{-kJ}, \ 2 \leq k \leq d} 
\end{align}
are polynomially modulated truncated singular kernels, and
\[ \mathbf{1}_{|\beta| \leq 2^{-2^{s/10A_0}}} \leq \chi_s \leq \mathbf{1}_{|\beta| \leq 2^{1-2^{s/10A_0}}} \]
is a smooth cut-off.

The key property of the $\{ L_{J,\vec{\lambda}}^s \}$ that we will use is encoded in the behavior of the Weyl sums, which satisfy
\[ \sup_{(\vec{A},B,Q) = 1} |S(\vec{A}/Q,B/Q)| \leq Q^{-c_d} \]
for some absolute $c_d > 0$ by standard Weyl sum estimates, see e.g.\ \cite[Appendix B]{BOOK}, but also stronger orthogonality estimates, see \cite[Lemma 6.3]{KSW}:

\begin{lemma}\label{l:maxest}
There exists an absolute $c=c_d > 0$ so that the following maximal estimate holds:
\begin{align}
\| \sup_{(\vec{A},Q) = 1, \ 2^{s-1} \leq Q < 2^s} |\int \sum_{B \leq Q} S(\vec{A}/Q,B/Q) \chi_s(\beta - B/Q) \hat{f}(\beta) e(\beta x)| \|_{\ell^2(\mathbb{Z})} \lesssim 2^{-cs} \|f \|_{\ell^2(\mathbb{Z})}.
\end{align}
\end{lemma}

With this quantitative input, we can effectively estimate $\mathcal{V}^r_{s,d} f$; the proof of the below estimate will occupy the remainder of the paper.

\begin{proposition}\label{p:final}
    There exists an absolute constant $0 < c=c_d < 1 $ so that following bound holds:
    \begin{align}
        \| \mathcal{V}^r_{s,d} f \|_{\ell^2(\mathbb{Z})} \lesssim_{\epsilon_0,\lambda}  (\frac{r}{r-2})^2 s^2 2^{-c s}  \|f \|_{\ell^2(\mathbb{Z})}.
    \end{align}
\end{proposition}

We let 
\[ I_s = [-2^{-2^{s/10A_0}},2^{-2^{s/10A_0}}] \]
denote the moral support of $\chi_s$; and we below we let $c=c_d > 0$ denote small constants, depending only on the degree, $d$; and so, we will often suppress the implicit dependence on $d$ in the estimates below.

With this in mind, the following two corollaries immediately present:

The first is a sequence-space estimate.
\begin{cor}\label{c:seqspace}
There exists $c > 0$ so that whenever $|I| \geq |I_s|$,
    \begin{align}
        \| \sup_{(\vec{A},Q) = 1, \ 2^{s-1} \leq Q < 2^s} | \sum_{B \leq Q} c_{B/Q} S(\vec{A}/Q,B/Q) e(B/Q x) | \|_{\ell^2(I)} \lesssim 2^{-c s} |I|^{1/2} \| c_{B/Q} \|_{\ell^2(\mathbb{Z})}.
    \end{align}
\end{cor}
\begin{proof}
Let $\mathbf{1}_I \leq |V_I|$ be a Schwartz function with spatial scale $I$ whose Fourier transform is supported in a $|I|^{-1}$ neighborhood of $0$, and restricting the supremum to
\[ {(\vec{A},Q) = 1, \ 2^{s-1} \leq Q < 2^s}, \]
bound
\begin{align}
&        \| \sup_{\vec{A}/Q} | \sum_{B \leq Q} c_{B/Q} S(\vec{A}/Q,B/Q) e(B/Q x) | \|_{\ell^2(I)} \\
& \leq \| \sup_{\vec{A}/Q} |\sum_{B \leq Q} c_{B/Q} V_I(x) S(\vec{A}/Q,B/Q) e(B/Q x) \|_{\ell^2(\mathbb{Z})} \\
& = \| \sup_{\vec{A}/Q} | \int e(\beta x) \sum_{B \leq Q} S(\vec{A}/Q,B/Q) \widehat{V_I}(\beta - B/Q) \big( \sum_{B \leq Q} c_{B/Q} \chi_s(\beta - B/Q) \big) \ d\beta \|_{\ell^2(\mathbb{Z})} \\
& = \| \sup_{\vec{A}/Q} |\int e(\beta x) \sum_{B \leq Q} S(\vec{A}/Q,B/Q) \chi_s(\beta - B/Q) \big( \sum_{B \leq Q} c_{B/Q} \widehat{V_I}(\beta - B/Q)  \big) \ d\beta \|_{\ell^2(\mathbb{Z})} \\
& \lesssim 2^{-cs} \| c_{B/Q} \|_{\ell^2(\mathbb{Z})} |I|^{1/2}.
\end{align}
\end{proof}

We record the following simple consequence of convexity in the below lemma. 

\begin{lemma}\label{c:convex}
    There exists $c>0$ so that whenever $\| \phi \|_{L^1(\mathbb{R})} \leq 1$
and \[
m_{\vec{\mu}} := \widehat{ \phi e(P_{\vec{\mu}}(\cdot))},\] 
the following estimate holds:
    \begin{align}
   \|  \sup_{(\vec{A},Q) = 1, \ 2^{s-1} \leq Q < 2^s, \ \vec{\mu}} |\int \sum_{B \leq Q} S(\vec{A}/Q,B/Q) (m_{\vec{\mu}} \chi_s)(\beta - B/Q) \widehat{f}(\beta) e(\beta x)|  \|_{\ell^2(\mathbb{Z})} \lesssim 2^{-cs} \|f \|_{\ell^2(\mathbb{Z})}.
    \end{align}
\end{lemma}

The proof of Proposition \ref{p:final} will largely derive from the below proposition concerning the following less oscillatory variation operators:
\begin{align}
        \mathcal{V}^r_sf(x) := \sup_{(\vec{A},Q) = 1, \ 2^{s-1} \leq Q < 2^s} \mathcal{V}^r \big( \int \sum_{B \leq Q} S(\vec{A}/Q,B/Q) \widehat{\Psi^s_J}(\beta - B/Q) \chi_s(\beta - B/Q) \hat{f}(\beta) e(\beta x) \ d\beta : J \big).
    \end{align}

\begin{proposition}\label{p:lowfreqvar00}
The following estimate holds for some $c > 0$:
\begin{align}
    \| \mathcal{V}^r_s f \|_{\ell^2(\mathbb{Z})} \lesssim_{\epsilon_0,\lambda} (\frac{r}{r-2})^2 s^2 2^{-cs} \| f \|_{\ell^2(\mathbb{Z})}.
\end{align}
\end{proposition}

To prove Proposition \ref{p:lowfreqvar00}, we linearize our supremum via  measurable functions $\vec{A}, Q: \mathbb{Z} \to \mathbb{Z}^{d-1} \times \mathbb{Z}_{\geq 1}$ and express
\begin{align}
    &\mathcal{V}^r_s f(x) \\
    &\equiv \mathcal{V}^r\Big(  \sum_{B \leq Q(x)} S(\vec{A}(x)/Q(x),B/Q(x)) \widehat{\Psi^s_{J}}(\beta - B/Q(x)) \chi_s(\beta - B/Q(x)) \hat{f}(\beta) e(\beta x) \ d\beta : J \Big).
\end{align}
We emphasize that $J \geq 2^{s/A_0}$ and all scales involved in the definition of $\Psi_J^s$ are similarly $\geq 2^{s/A_0}$, and that our estimates below will be independent of this particular choice of linearization. With this in mind, for notational we will suppress the superscript $\Psi_J^s \longrightarrow \Psi_J$, and will similarly suppress the dependence on ${\epsilon_0,\lambda}$.

\begin{proof}[Proof of Proposition \ref{p:lowfreqvar00}]
Set $F_{B/Q}(x) := \chi_s * (\text{Mod}_{-B/Q} f)(x)$ so that we may express the foregoing as variation with respect to the operator
\begin{align}
\sum_{B \leq Q, \ 2^{s-1} \leq Q < 2^s} e(B/Q x) \Psi_j * F_{B/Q}(x) \cdot S(\vec{A}(x)/Q(x),B/Q) \mathbf{1}_{B/Q \in \{ B/Q(x) : B \leq Q(x) \}}.
\end{align}
Note that whenever $|I| = |I_s|$ is an interval, for any $x_I, y_I \in I$ we may Taylor expand
\begin{align}
&\sum_{B \leq Q, \ 2^{s-1} \leq Q < 2^s} e(B/Q x) \Psi_j * F_{B/Q}(x) S(\vec{A}(x)/Q(x),B/Q) \mathbf{1}_{B/Q \in \{ B/Q(x)\}} \\
&= \sum_{B \leq Q, \ 2^{s-1} \leq Q < 2^s} e(B/Q x) \Psi_j * F_{B/Q}(x_I) S(\vec{A}(x)/Q(x),B/Q) \mathbf{1}_{B/Q \in \{ B/Q(x)\}} + O(2^{-10 s} M_{\text{HL}} f(y_I))
\end{align}
since all scales involved satisfy $j \geq 2^{s/A_0}$.

With $x_I$ fixed, we apply metric chaining to the set
\begin{align}
X(x_I) := \{ \big( \Psi_j*F_{B/Q}(x_I) \big)_{B \leq Q, \ 2^{s-1} \leq Q < 2^s} : j \geq 2^{s/A_0}\}.
\end{align}
In particular, let
\begin{align}
    \vec{N}_\lambda(x_I)
\end{align}
denote the $\ell^2(B/Q)$-jump-counting function associated to the set $X(x_I)$, namely
\begin{align}
    \vec{N}_\lambda(x_I) := &\sup\{ K : \text{ there exists } k_0 < k_1 < \dots < k_K \text{ so that} \\
    & \| \Psi_{k_i} * F_{B/Q}(x_I) - \Psi_{k_{i-1}} * F_{B/Q}(x_I) \|_{\ell^2(B/Q)} \geq \lambda \},
\end{align}
let
\begin{align}
    \mathcal{F}_s f(x) := \sup_{j \geq 2^{s/A_0}} \big( \sum_{B \leq Q, \ 2^{s-1} \leq Q < 2^s} |\Psi_j*F_{B/Q}(x)|^2 \big)^{1/2}
\end{align}
and
\begin{align}
    \mathcal{V}^r f(x) := \sup \big( \sum_{i} \| \Psi_{j_i}*F_{B/Q}(x_I) - \Psi_{j_{i+1}}*F_{B/Q}(x) \|_{\ell^2(B/Q)}^r \big)^{1/r}
\end{align}
where the supremum runs over all finite increasing subsequences. Note that
\begin{align}\label{e:bounds}
    \| \mathcal{F}_sf \|_{\ell^2(\mathbb{Z})} + \frac{r-2}{r} \| \mathcal{V}^r f\|_{\ell^2(\mathbb{Z})} \lesssim \| f \|_{\ell^2(\mathbb{Z})}
\end{align}
by standard singular integral estimates, see \cite[Theorem 1.2]{JSW}, and Magyar-Stein-Wainger Transference.

Now, for each $v$ so that 
\[ 2^{-v} \leq \text{diam}(X(x_I)) \leq 2 \mathcal{F}_sf(x_I),\]
define $\Lambda_v(x_I)$ to be a collection of times $t$ so that
\begin{align}
    X(x_I) \subset \bigcup_{t \in \Lambda_v(x_I)} \{ (b_{B/Q})_{B/Q} : \| b_{B/Q} - \Psi_t*F_{B/Q}(x_I) \|_{\ell^2(B/Q)} \leq 2^{-v} \},
\end{align}
and for each $t \in \Lambda_v(x_I)$ define the \emph{parent} of $t$, $\varrho(t) \in \Lambda_{v-1}(x_I)$ to be the minimal time so that
\begin{align}
    &\{ (b_{B/Q})_{B/Q} : \| b_{B/Q} -\Psi_t*F_{B/Q}(x_I) \|_{\ell^2(B/Q)} \leq 2^{-v} \} \\
    & \qquad \cap 
    \{ (b_{B/Q})_{B/Q} : \| b_{B/Q} - \Psi_{\varrho(t)}*F_{B/Q}(x_I) \|_{\ell^2(B/Q)} \leq 2^{1-v} \} \neq \emptyset.
\end{align}
Set
\begin{align}
    \nu_{t} := \Psi_t - \Psi_{\varrho(t)}.
\end{align}

Then for $x,y \in I$
\begin{align}
    &\mathcal{V}^rf(x) \leq \sum_{2^{-v} \leq 2 \mathcal{F}_s(x_I)} \mathcal{V}^r\Big(  \sum_{B \leq Q(x)} e(B/Q(x) x) S(\vec{A}(x)/Q(x),B/Q(x)) \nu_t*F_{B/Q(x)}(x_I)  : t \in \Lambda_v(x_I)  \Big) \\
    & \qquad + 2^{-10s} M_{\text{HL}} f(y) \\
    & \leq \sum_{2^{-v} \leq 2 \mathcal{F}_s(x_I)} \big( \sum_{t \in \Lambda_v(x_I)} | \sum_{B \leq Q(x)} e(B/Q(x) x) S(\vec{A}(x)/Q(x),B/Q(x)) \nu_t*F_{B/Q(x)}(x_I)|^r \big)^{1/r}  \\
    & \qquad + 2^{-10s} M_{\text{HL}} f(y).
\end{align}
The argument now concludes as in the proof of \cite[Theorem 1.4]{KJump}, with the key estimate being
\begin{align}\label{e:minest}
&\| \big( \sum_{t \in \Lambda_v(x_I)} | \sum_{B \leq Q(x)} e(B/Q(x) x) S(\vec{A}(x)/Q(x),B/Q(x)) \nu_t*F_{B/Q(x)}(x_I)|^r \big)^{1/r} \|_{\ell^2(I)} \\
& \lesssim  2^{-cs} 2^{-v} |I|^{1/2} \min\{ 2^s \vec{N}_{2^{-v}}(x_I)^{1/r}, \vec{N}_{2^{-v}}(x_I)^{1/2} \};
\end{align}
the first estimate follows Cauchy-Schwartz and pointwise considerations, 
noting that 
\begin{align}
|\{ B \leq Q \leq 2^{s} \}| \approx 2^{2s},
\end{align}
while the second estimate follows from replacing the $\ell^r$ sum with the stronger $\ell^2$ sum, and then applying Corollary \ref{c:seqspace}.
\end{proof}

With this in hand, we can quickly prove Proposition \ref{p:final}.

\begin{proof}[The Proof of Proposition \ref{p:final}]
We will use the \emph{polynomial coefficient norm},
\[ \| P \| := \sum_{j \geq 1} |\lambda_j|, \; \; \; P(t) := \sum_{j\geq 0} \lambda_j t^j \in \mathbb{R}[\cdot] \]
to organize our analysis.

Let $A_1$ be a large constant, and notice that for each $\vec{\mu}$, note that there are only $O(d^2 A_1)$ many scales $j$ so that there exists $2 \leq k \neq k' \leq d$
\begin{align}
    2^{-A_1} \leq \frac{ |\mu_k| 2^{jk}}{|\mu_{k'}| 2^{jk'}} \leq 2^{A_1};
\end{align}
collect these scales in the set $\mathcal{J}_0(\vec{\mu})$.

For each $\vec{\mu} \in [2^{-10s}]^{d-1}$, we partition our set of scales into three sets: let
\begin{align}
    \mathcal{J}_{\leq}(\vec{\mu}) &:= \{ j \notin \mathcal{J}_0(\vec{\mu}): \| P_{\vec{\mu}}(2^j \cdot) \| \leq 2^{-A_1 s} \} \\
    \mathcal{J}_{\approx}(\vec{\mu}) &:= \mathcal{J}_0(\vec{\mu}) \cup \{ j : 2^{-A_1 s} \leq \| P_{\vec{\mu}}(2^j \cdot) \| \leq 2^{A_1 s} \} \; \; \; \text{ and} \\
    \mathcal{J}_{\geq}(\vec{\mu}) &:= \{ j \notin \mathcal{J}_0(\vec{\mu}): \| P_{\vec{\mu}}(2^j \cdot) \| \geq 2^{A_1 s} \}.
\end{align}
Note that for each $j \notin \mathcal{J}_{0}(\vec{\mu})$, there exists a unique $k = k(j)$ so that
\begin{align}
    \| P_{\vec{\mu}}(2^j \cdot) \| \approx |\mu_k| 2^{kj},
\end{align}
so in particular 
\[ \sup_{\vec{\mu}} |\mathcal{J}_{\approx}(\vec{\mu})| \lesssim_d s \]
and $\mathcal{J}_{\leq}(\vec{\mu}), \mathcal{J}_{\geq}(\vec{\mu})$ have at most $O_d(1)$ many connected components. More to the point, if we let
\begin{align}
    \mathcal{J}_l(\vec{\mu}) := \{ j \notin \mathcal{J}_0(\vec{\mu}) : \| P_{\vec{\mu}}(2^j \cdot) \| \approx 2^l \}
\end{align}
then
\begin{align}
    \sup_{\vec{\mu}, l \geq - A_1 s} |\mathcal{J}_l(\vec{\mu})| \lesssim_d 1.
\end{align}

With this in mind, set
\[ \phi_{j,\vec{\mu}}(t) := \psi_j(t) e(P_{\vec{\mu}}(t)), \]
and, abbreviating
\[ \sup_{(\vec{A},Q)= 1, \ 2^{s-1} \leq Q < 2^s}  \longrightarrow \; \; \; \sup_{\vec{A}/Q},\]
bound
\begin{align}
    &\mathcal{V}^r_{s,d} f(x) \\
    & \leq 
\mathcal{V}^r_s f(x) + \sup_{\vec{A}/Q,  \vec{\mu}} \sum_{B \leq Q} \sum_n |f(x-n)| \Big( \sum_{j \in \mathcal{J}_{\leq}(\vec{\mu})} \int |\chi_s^{\vee}(n-t)| |\psi_j(t)| |e(P_{\vec{\mu}}(t)) - 1| \ dt  \Big) \\
& + \sup_{\vec{A}/Q, \vec{\mu}} \sum_{j \in \mathcal{J}_{\approx}(\vec{\mu})} |\int \sum_{B \leq Q} S(\vec{A}/Q,B/Q) \phi_{j,\vec{\mu}}(\beta - B/Q) \chi_s(\beta - B/Q) \hat{f}(\beta) e(\beta x) \ d\beta| \\
    & + \sum_{l \geq A_1 s} \sup_{\vec{A}/Q, \vec{\mu}} \sum_{j \in \mathcal{J}_l(\vec{\mu})} |\int \sum_{B \leq Q} S(\vec{A}/Q,B/Q) \phi_{j,\vec{\mu}}(\beta - B/Q) \chi_s(\beta - B/Q) \hat{f}(\beta) e(\beta x) \ d\beta| \\
    & \leq \mathcal{V}^r_s f(x) + 2^{-10s} M_{\text{HL}} f(x) \\
&    + \sup_{\vec{A}/Q, \vec{\mu}} \sum_{j \in \mathcal{J}_{\approx}(\vec{\mu})} |\int \sum_{B \leq Q} S(\vec{A}/Q,B/Q) \phi_{j,\vec{\mu}}(\beta - B/Q) \chi_s(\beta - B/Q) \hat{f}(\beta) e(\beta x) \ d\beta| \\
& + \sum_{\vec{A}/Q} \sum_{B \leq Q} \sum_{l \geq A_1 s} \sup_{\vec{\mu}} \sum_{j \in \mathcal{J}_l(\vec{\mu})} |\int \phi_{j,\vec{\mu}}(\beta - B/Q) \chi_s(\beta - B/Q) \hat{f}(\beta) e(\beta x) \ d\beta|;     
\end{align}
in the first step we have expanded
\begin{align}
    \psi_j(t) e(P_{\vec{\mu}}(t)) = \psi_j(t) + \psi_j(t) \big( e(P_{\vec{\mu}}(t)) - 1 \big).
\end{align}
The first term is governed by Proposition \ref{p:lowfreqvar00}, the sum over $\mathcal{J}_{\approx}(\vec{\mu})$ contribute $O_d(s)$ many scales, so by Lemma \ref{c:convex} the total contribution from these scales is $O(s  2^{-cs})$. It remains only to address the contribution from the scales in $\mathcal{J}_l(\vec{\mu})$; by conceding a constant factor, we can assume for that each $\vec{\mu}$, there exists a unique such scale, call it $j_l(\vec{\mu})$. By modulating, it suffices to estimate
\begin{align}
    \| \sup_{\vec{\mu}} |\int (\chi_s*f)(x-t) \psi_{j_l(\vec{\mu})}(t) e(P_{\vec{\mu}}(t)) \ dt| \|_{\ell^2(\mathbb{Z})};
\end{align}
by Magyar-Stein-Wainger transference and $TT^*$, see \cite[\S 14]{BOOK}, there exists an absolute constant $c = c_d  > 0$ so that, uniformly in $l \geq 1$,
\begin{align}
        \| \sup_{\vec{\mu}} |\int (\chi_s*f)(x-t) \psi_{j_l(\vec{\mu})}(t) e(P_{\vec{\mu}}(t)) \ dt| \|_{\ell^2(\mathbb{Z})} \lesssim 2^{-cl} \|f \|_{\ell^2(\mathbb{Z})},
\end{align}
so
\begin{align}
    &\| \sum_{\vec{A}/Q} \sum_{B \leq Q} \sum_{l \geq A_1 s} \sup_{\vec{\mu}} \sum_{j \in \mathcal{J}_l(\vec{\mu})} |\int \widehat{ \psi_j e(P_{\vec{\mu}}(\cdot)) }(\beta - B/Q) \chi_s(\beta - B/Q) \hat{f}(\beta) e(\beta x) \ d\beta)| \|_{\ell^2(\mathbb{Z})} \\
    & \lesssim 2^{sd} \sum_{l \geq A_1 s} 2^{-c l} \| f \|_{\ell^2(\mathbb{Z})} \leq 2^{-s} \|f \|_{\ell^2(\mathbb{Z})},
\end{align}
provided $A_1$ is chosen sufficiently large. Putting everything together, we have bounded
\begin{align}
    \| \mathcal{V}^r_{s,d} f \|_{\ell^2(\mathbb{Z})} \lesssim \big( (\frac{r}{r-2})^2 s^2 2^{-cs} + s 2^{-cs} + 2^{-s} \big) \|f \|_{\ell^2(\mathbb{Z})},
\end{align}
as desired.
\end{proof}

\end{document}